\documentclass[11pt]{amsart}
\usepackage[english]{babel}
\usepackage{abraces}
\usepackage{amsmath,amsthm,amsfonts,amssymb}
\usepackage[foot]{amsaddr}
\usepackage{mathtools}
\usepackage{color}

\usepackage{mathrsfs}
\usepackage[all]{xy}
\usepackage{tikz-cd}
\usepackage[normalem]{ulem} 
\usepackage{datetime}

\usepackage{hyperref}

\theoremstyle{definition}

\newtheorem{lemma}{Lemma}[section]

\newtheorem{remark}{Remark}[section]
\newtheorem{proposition}{Proposition}[section]

\calclayout

\title{Something}

\date{\today}
\def\div{\mathrm{div}\,}

\def\u{{u}}

\def\zero{{0}}
\newcommand{\vc}[1]{{#1}}
\def\v{{v}}

\def\n{{n}}
\def\dps{\displaystyle}
\def\R{{\mathbb R}}

\def\S{{\mathbb S}}
\def\CC{{\mathscr C}}
\def\D{{\mathscr D}}
\def\pt{\partial}

\def\<{\langle}
\def\>{\rangle}
\def\inclus{{\hookrightarrow}}

\def\bve{\;|\;}

\def\fleche{\rightarrow}
\def\eps{\epsilon}

\def\x{{ x}}
\def\n{{n}}
\def\vh{\vc{\eta}}

\def\wei{\varrho}
\def\grad{\nabla}
\def\div{{\rm div}\;}

\def\pt{\partial}
\def\grad{\nabla}

\def\eps{\varepsilon}
\def\ominft{\Omega^\infty}
\def\oms{\Omega_\star}
\def\omz{{\Omega}_0}

\def\bomega{\overline{\omega}}
\def\bomz{\overline{\Omega}_0}
\def\boms{\overline{\Omega}_\star}
\def\bominft{\overline{\Omega}_\infty}

\newcommand{\hyp}[1]{(${\mathscr H}_{#1}$)}
\newcommand{\Wei}[3]{W^{#1}_{#2}(#3)}
\newcommand{\Wlog}[3]{W^{#1}_{{{\rm log}}}(#3)}

\newcommand{\WL}[1]{L^2_\wei(#1)}
\newcommand{\WLINV}[1]{L^2_{1/\wei}(#1)}
\newcommand{\Hdivw}[1]{H_{1/\wei}({\rm{div}};#1)}
\newcommand{\nrm}[1]{|{#1}|}
\def\bom{\overline{\Omega}}

\def\a{a}

\def\Tinft{T}

\def\xs{x_\star}

\def\cst{c^\star}
\def\Ks{K_\star}
\def\hv{{\hat{v}}}

\def\bw{{w}}
\def\hbw{{\widehat{\bw}}}
\def\bws{w^\star}
\def\hbws{{\widehat w}^\star}

\def\Tauh{{\mathcal T}}
\def\Taush{{\mathcal T}^{\star}}
\def\Taussh{{\mathcal T}^{\star\star}}

\def \psihat{\widehat{\psi}}
\def\Js{J^\star}

\def\tJs{{\pt J^{\star}}}

\def\xstar{x_{\star}}

\def\Matstar{G}
\def\bfs{k}
\def\cfs{V}
\newcommand{\xst}[1]{ x^{\star}_{#1}}

\renewcommand{\leq}{\leqslant}
\renewcommand{\geq}{\geqslant}

\def\om{{\omega}_{\rm ext}}

\def\bom{\overline{\omega}}
\def\bome{\overline{\om}}

\title[Inverted finite elements for exterior Neumann problem]{Inverted finite elements approximation of
the Neumann problem for  second order
elliptic equations  in exterior two-dimensional domains}

 \author{R. Belbaki$^4$}
\address{\rm  $^4$ Ecole Normale Supérieure de Kouba, Algiers, Algeria. }
\author{S. K. Bhowmik$^3$}
\address{\rm  $^3$ Department of Mathematics, University of Dhaka,  Bangladesh}
\author{T. Z. Boulmezaoud$^2$}
\address{\rm  $^2$ Universit\'e Paris-Saclay, UVSQ, LMV, Versailles, France.}
\author{N. Kerdid$^1$}
\address{\rm $^1$ IMSIU, College of Sciences,  Department of mathematics and Statistics, P. O. Box 90950, Riyadh, 11623, Saudi Arabia.}
\author{S. Mziou$^5$}
\address{\rm  $^5$ Laboratoire Lammda, ESSTHS, Universit\'e de Sousse, Tunisia}

\begin{document}

\begin{abstract}
We use  inverted finite elements method  for  approximating solutions of  second order elliptic equations with non-constant coefficients varying to infinity   in the exterior of a 2D bounded obstacle, when a Neumann boundary condition is considered.  After proposing an appropriate functional framework for the deployment of the method, we analyse its convergence and detail its implementation.  Numerical tests performed after implementation confirm  convergence and  high efficiency of the method. 
 \end{abstract}

{\let\newpage\relax\maketitle}
\section{Introduction}
In this paper we deal with the approximation of the second order elliptic equation
\begin{equation}\label{second_ord_equa}
- \div(\sigma  \grad u)  = f \; \mbox{ in } \R^2  \backslash \overline{\omega},
\end{equation}
when a Neumann boundary condition is imposed on its boundary
\begin{equation}\label{neumann_code}
\sigma \grad u . \n   = g \;  \mbox{ on } \pt \omega. 
\end{equation}
Here $\omega$ designates a bounded open set of $\R^2$ having a lispchitzian boundary $\pt \omega$,  
$ \R^2  \backslash \overline{\omega}$ is its exterior and $\n$ is the unit normal vector 
pointing towards $\omega$ and  outwards $ \R^2  \backslash \overline{\omega}$. 
Formally, we also complete equation \eqref{second_ord_equa} with
a condition at infinity which we momentarily write as 
\begin{equation}\label{cond_behav1}
u(\x) = \circ (\log(|\x|)^{1/2}) \mbox{ when } \nrm{\x} \fleche +\infty.
\end{equation}
The exact meaning of this condition will be clarified later. Actually, since the domain $ \R^2  \backslash \overline{\omega}$  is unbounded, appropriate weighted spaces
will be used  for describing the decay or the growth of functions at infinity. \\
Following  \cite{boulmezaoudm2an}, \cite{boulmezaoud_mziou1}, \cite{boulmezaoud_mziou2},  and \cite{bhowmikmziou}, our approach  consists to use Inverted Finite Elements  Method (IFEM) introduced by one of the authors \cite{boulmezaoudm2an}. This method was used with a significant success in solving 2D Dirichlet's problem
even when the coefficient $\sigma$  is indefinitely  varying  (see \cite{bhowmikmziou}). It has also been used to solve one-dimensional problems, with or without singularities (see \cite{boulmezaoud_mziou1} and \cite{boulmezaoud_mziou2}), and fully three-dimensional problems (see, e. g., \cite{boulmezaoudm2an} and \cite{boulmezaoud_divcurl_ifem}). 
This short paper is in the continuation of paper  \cite{bhowmikmziou}  in which the authors use IFEM for solving the same equation with a Dirichlet's boundary condition. Although the approach is somewhat similar, there are substantial differences between the two systems of equations as well as important technical details about the method, which fully justify this work. \\
$\;$\\
The outline of this paper is as follows. In Section \ref{funcspa} we introduce an appropriate family
of weighted spaces and we prove the well posedness of the problem. A weak formulation
of equation \eqref{second_ord_equa}-\eqref{neumann_code} is also given.
In Section \ref{ifem_method}, inverted finite element method is adapted to the problem.
The last section is devoted to some computational results which prove
the efficiency of the method.

\section{Function spaces and variational formulation.}\label{funcspa}
Our intention here is is to set up and establish an adequate formulation of the 
system \eqref{second_ord_equa}-\eqref{neumann_code} 
 that is suitable for its discretization by the inverted finite elements method.  In the sequel, we set
 $$
 \om = \R^2 \backslash \bom. 
 $$
Following  \cite{hanouzet} and \cite{boulmezaoud2, boulmezaoud3},
we  base our approach on particular weighted
 Sobolev spacse. Since the $2$-dimensional
case is often considered as critical (see, e. g., \cite{giroire}), the  weight we use contains
a logarithmic factor. Set
$$
\wei(\x) =   (|\x|^2 + 1)^{-1} (\log(2+|\x|^2))^{-2}, \; \x \in \R^2,
$$
and define the space $\WL{\om}$ of all the measurable functions on $\om$ satisfying
$$
\int_{\om}  \wei(\x)  |v|^2 dx = \int_{\om} \frac{ |v|^2 }{(|\x|^2 + 1) (\log(2+|\x|^2))^2}  dx < +\infty.
$$
This weighted space is equipped  with the natural norm
$$
\|v\|_{L^2_\wei(\om)} = \left(\int_{\om} \wei(\x) |v|^2 dx\right)^{1/2}.
$$
Similarly, let $\WLINV{\om}$ be the space of (generalized)  functions satisfying
$$
\int_{\om} \frac{ |v|^2 }{\wei(\x)}dx =   \int_{\om} (|\x|^2 + 1) (\log(2+|\x|^2))^2  |v|^2  dx < +\infty.
$$
Obviously,
$$
\WLINV{\om} \inclus L^2(\om) \inclus \WL{\om}.
$$
For each function $\lambda \in \WLINV{\om}$, the map $\varphi_\lambda  :v \in \WL{\om} \mapsto \int_{\om} \lambda  v dx$ belongs
to $\WL{\om}'$, the dual of $\WL{\om}$ (thanks to Cauchy-Schwarz inequality). Thus, in view of  Riesz's theorem  the map $ \lambda  \in \WLINV{\om} \mapsto  \varphi_\lambda  \in \WL{\om}'$ is an isometry. In other words,
  $\WLINV{\om}$ is a realization of the dual  $\WL{\om}'$ with $L^2(\om)$ as a pivot space. Thus, in what follows
   $\WL{\om}'$  will be identified to  $\WLINV{\om}$. \\

Consider now the weighted Sobolev space $\Wlog{1}{0}{\om}$ composed of the all the measurable functions on $\om$ satisfying
\begin{equation}\label{asympto_cond}
v \in  L^2_\wei(\om), \; \grad v \in L^2(\om)^2,
\end{equation}
and endowed with the norm
$$
\|v\|_{\Wlog{1}{0}{\om}} =  \left(\|  v \|_{L^2_\wei(\om)}^2 + \| \grad v\|_{L^2(\om)^2}^2 \right)^{1/2}.
$$
We also consider the space $\Hdivw{\om}$ of all the vector fields $\v \in L^2(\om)^2$ satisfying 
$$
\wei^{-1/2} \div \v \in L^2(\om).
$$
We also set 
$$\D(\bome) = \{ u_{\bome} \bve u \in \D(\R^2) \}$$
where  $\D(\R^2)$ is the space of infinitely differentiable functions on $\R^2$ with a compact support.  The space $\D(\bome)$ is dense in $\Wlog{1}{0}{\om}$
and similarly $\D(\bome)^2$ is dense in $\Hdivw{\om}$. \\
 It can be proven that  when $|\x| \fleche +\infty$, functions of $\Wlog{1}{0}{\om}$
have a decaying behaviour towards $0$  or at most a logarithmic growth.
More precisely, for $v \in \Wlog{1}{0}{\om}$ we have (see \cite{bhowmikmziou} for the proof):
\begin{equation}\label{inequa_logr}
\lim_{r \fleche +\infty} \frac{\|v(r, \cdot)\|_{L^2(\S^1)}}{ \sqrt{\log r}} = 0,
\end{equation}
where $\S^1$ is the unit circle of $\R^2$ and
$$
\|v(r, .)\|_{L^2(\S^1)} = \left(\int_{\S^1} v(r, \sigma)^2 d\sigma\right)^{1/2}.
$$
In other words,
$$
|v(\x)| = \circ ((\log|\x|)^{1/2}) \mbox{ when } |\x| \fleche +\infty.
$$
We may also observe that functions of
 $\Wlog{1}{0}{\om}$ and $\Hdivw{\om}$ have  locally the same nature as functions
 of the usual Sobolev spaces $H^1$ and $H({\rm{div}})$ respectively (see, e. g., \cite{girault0}). Actually,
 if $K$ stands for an arbitrary compact subset of $\bome$, then, for all         
 $v \in \Wlog{1}{0}{\om}$, $v_{|K} \in H^1(K)$ where $H^1(K)$ denotes the usual
  Sobolev space defined over $K$.  Similarly, for $\v \in \Hdivw{\om}$,
  $\v_{|K} \in H({\rm{div}}; K)$, where $ H({\rm{div}}; K) = \{ \v \in L^2(K)^2 \bve \div \v \in  L^2(K)\}$
   (see, e. g., \cite{girault0}). The difference is indeed in the asymptotic behavior
   when $|\x| \fleche +\infty$.   In view of this, and  since $\pt \omega$ is bounded,
   the trace operator $\gamma_0\; :\;  v \in \Wlog{1}{0}{\om}  \mapsto v_{|\pt \omega} \in H^{1/2}(\pt \omega)$ is well defined and continuous. Similarly, the normal trace operator $$\gamma_N\; :\; \v  \in   \Hdivw{\om} \mapsto  \v.\n \in H^{-1/2}(\pt \omega)$$
is also well  defined, continuous, onto and 
the following Green's formula holds true for $\v \in  \Hdivw{\om}$ and $\varphi \in \Wlog{1}{0}{\om}$: 
\begin{equation}\label{weighted_green_2d}
 \int_{\om} \v. \grad \varphi dx = - \int_{\om}  \div \v.  \varphi dx + \< \v.\n, \gamma_0 \varphi\>_{H^{-1/2}(\pt \omega), H^{1/2}(\pt \omega)}.
\end{equation}
Now, we need the following standard 
assumptions
\begin{itemize}
\item[\hyp{1}] $\sigma  \in L^\infty(\om)$ and there exists a constant $\sigma_0 > 0$ such that
$$
\sigma(\x) \geq \sigma_0 > 0, \ae \mbox{ in } \om.
$$
\item[\hyp{2}] $f \in \WLINV{\om}$ and $g \in H^{-1/2}({\pt \omega})$ are such that
\begin{equation}\label{compatibility}
\int_{\om} f dx + \<g, 1\>_{\pt \omega} = 0,
\end{equation}
\end{itemize}
and we look for solutions of   \eqref{second_ord_equa}-\eqref{neumann_code} belonging to $\Wlog{1}{0}{\om}$. 
However, we may notice that, as in a bounded domain, but not exactly for the same reasons,  uniqueness is lost for the Neumann problem
  \eqref{second_ord_equa}-\eqref{neumann_code}. Actually, among polynomial functions only constant ones  belong to   $\Wlog{1}{0}{\om}$. Under similar asymptotic conditions, this
  loss of uniqueness  occurs also in 1D situations (see, e. g., \cite{boulmezaoud_mziou1} and \cite{boulmezaoud_mziou2}), but does not appear in 3D  (see, e. g., \cite{boulmezaoud2, boulmezaoud1},  \cite{boulmezaoudm2an}). 
   For this reason, we complete
   system  \eqref{second_ord_equa}-\eqref{neumann_code} with the following condition
\begin{equation}\label{mean_cond}
 \int_{\om} \frac{ u }{((|\x|^2 + 1) (\log(2+|\x|^2)^2} dx =0.
\end{equation}
This condition makes sense since  $\wei u \in L^1(\om)$. It also means that $u$ is orthogonal to constant functions
with respect to the inner product $(.,.)_{L^2_\wei(\om)}$. 
 We have
 \begin{proposition}\label{proof_propo_exist}
Under assumptions \hyp{1}-\hyp{2},  system \eqref{second_ord_equa}-\eqref{neumann_code}-\eqref{mean_cond} has one and only one solution $u \in \Wlog{1}{0}{\om}$. Moreover,
\begin{equation}
\| u \|_{\Wlog{1}{0}{\om} }  \leq C( \| f\|_{\WLINV{\om} }
 + \| g\|_{H^{-1/2}{(\pt \omega)}}),
\end{equation}
for some constant $C > 0$ depending only on $\omega$ and the coefficient $\sigma$.
\end{proposition}
\begin{proof}[Proof of Proposition \ref{proof_propo_exist}]
The proof is based on the following variational formulation of the problem: 
\begin{lemma}\label{lemma_prth}
Under the assumptions \hyp{1}-\hyp{2},
$u \in \Wlog{1}{0}{\om}$  is solution of \eqref{second_ord_equa}-\eqref{neumann_code}-\eqref{mean_cond}
iff: for all $v \in \Wlog{1}{0}{\om}$
\begin{equation}\label{weak_form_n}
\int_{\om} \sigma \grad u. \grad w dx +  \left(\int_{\om} \wei(\x) u  dx\right)\left(\int_{\om} \wei(\x) w  dx\right) =  \int_{\om} f w dx +   \<g, w\>_{\pt \omega}.
\end{equation}
\end{lemma}
\begin{proof}[Proof of Lemma \ref{lemma_prth}]
If $u$ is solution of  \eqref{second_ord_equa}-\eqref{neumann_code}-\eqref{mean_cond}, then
multiplying  \eqref{second_ord_equa} by $v$ and using Green's formula \eqref{weighted_green_2d} with $\v = \grad u$ and $\varphi = w$ gives \eqref{weak_form_n}. Conversely, suppose that $u$ satisfies \eqref{weak_form_n}. Choosing $v=1 \in  \Wlog{1}{0}{\om}$ and using
\eqref{compatibility} gives \eqref{mean_cond}. Then, we prove that $u$ satisfies  \eqref{second_ord_equa} and \eqref{neumann_code}
in a usual manner.
\end{proof}

Well posedness of  \eqref{second_ord_equa}-\eqref{neumann_code}-\eqref{mean_cond} is a consequence of Lax-Milgram theorem. Actually, continuity of the bilinear form on LHS and the linear form on the RHS  of \eqref{weak_form_n} 
results from  assumptions \hyp{1}-\hyp{2}. Coercivity follows from 
the following  Hardy  type inequality (see \cite{giroire}):   there exists
a constant $C > 0$ such that for all $v \in \Wlog{1}{0}{\om}$ 
\begin{equation}
\inf_{\lambda \in \R}  \int_{\om} \frac{|v - \lambda|^2}{(|x|^2+1) (\log(|x|^2+2))^2} dx  \leq C \int_{\om}  | \grad v|^2 dx.
\end{equation}
\end{proof}

\section{Inverted Finite Element Method. Discretization. }\label{ifem_method}
In this section, focus is on the discretization of the system  \eqref{second_ord_equa}-\eqref{neumann_code}-\eqref{mean_cond}  by means of inverted finite element method (IFEM), which was introduced by one of the authors in \cite{boulmezaoudm2an}. The way we unroll the method here
is inspired from \cite{bhowmikmziou} where the Dirichlet's problem is considered. Therefore, only the main elements of the method will
be introduced here. The reader can refer to \cite{bhowmikmziou} for more details (see also   \cite{boulmezaoudm2an, boulmezaoud_mziou1, boulmezaoud_mziou2}). \\
The starting idea  consists to introduce a domain decomposition of $\om$ into the form
\begin{equation}\label{decompo_inft}
\bome = \bomz  \cup (\cup_{i=1}^M \Tinft_i),
\end{equation}
where
\begin{itemize}
\item $\omz$ is a bounded open subset surrounding $\omega$,
\item $\Tinft_1,\cdots,\Tinft_M$ are $M$ (big) infinite triangles satisfying  the assumptions (see Figure \ref{decomp_illust}): 
\begin{itemize}
\item  $\Tinft_1,\cdots,\Tinft_M$ have the origin as a  {\it common} fictitious vertex,
\item $\mathring{\Tinft_i} \cap \omz = \emptyset$, for $1 \leq i \leq M$,
\item For $i \ne j$, $\Tinft_i \cap \Tinft_j$ is either a common infinite edge or empty.
\end{itemize}
\end{itemize}
The term {\it infinite triangle} means  a geometric unbounded closed set of the form
$$
\Tinft = \{ \lambda_0 \a_0 +  \lambda_1 \a_1+  \lambda_ 2\a_2 \bve \lambda_0 \leq 0, \lambda_1 \geq 0, \lambda_2 \geq 0, \; \lambda_0 + \lambda_1 + \lambda_2=1\}
$$
where $\a_0$, $\a_1$ and $\a_2$ are three non-aligned points. Points  $\a_1$ and $\a_2$  are called the real vertices of $\Tinft $, while $ \a_0$ is called the fictitious vertex.  
We set
$$
\ominft  = \aoverbrace[L1R]{\Tinft_1 \cup \cdots \cup \Tinft_M}[U]^{\circ},
$$
and
$$
\oms = \R^2 \backslash  {\bominft} =  \aoverbrace[L1R]{ \omega \cup \omz}[U]^{\circ} 
$$
The  domain $\oms$ is a fictitious {\it bounded} open set which will play a prominent role in the method.  \\
$\;$\\
In the sequel, for $1 \leq i \leq M$, $S_i$  designates the triangle whose vertices
are $\zero$, $\a_{1}^{(i)}$ and $\a_{2}^{(i)}$, where $\a_{1}^{(i)}$ and $\a_{2}^{(i)}$ are the real vertices
of $\Tinft_i$. Thus, 
$$
S_1 \cup \cdots \cup S_M =\boms, \;  \cup_{i=1}^M (S_i \cup \Tinft_i) = \R^2, 
$$
and 
$$
\bomz \cap \bominft = \boms \cap \bominft = \cup_{i=1}^M (S_i \cap \Tinft_i).
$$
For each $i \leq M$, let $r_i$  be the unique {\it affine} function defined over $S_i \cup \Tinft_i$ satisfying $r_i(\zero)=0$ and $r_i(\a_{1}^{(i)}) = r_i(\a_{2}^{(i)}) = 1$. It can be easily seen that
$r_i(\x) \geq 1$ for $\x \in \Tinft_i$ and $0 \leq r_i(\x) \leq 1$ for $\x \in S_i$. The {\it global
radius} is defined on $\R^2$ as follows
\begin{equation}
r(\x) = r_i(\x) \mbox{ if } \x \in S_i \cup \Tinft_i, \; i = 1, \cdots, M.
\end{equation}
This function $\x \in \R^2 \mapsto r(\x)$ is continuous on $\R^2$ (see \cite{boulmezaoudm2an}).  Moreover, there exists
two constants $c_1 > 0$ and $c_2 > 0$ such that (see \cite{boulmezaoudm2an}) 
\begin{equation}\label{estima_ray}
c_1 |\x| \leq r(\x) \leq c_2 |\x|, \;  \mbox{ for all } \x \in \R^2. 
\end{equation}
Define also the polygonal inversion $\Phi \; :\;  \bominft \mapsto \boms$ by
$$
\; \Phi(\x) = \frac{\x}{r(\x)^2}, \;  \mbox{ for } \x \in \bominft. 
$$
The transformation  $\Phi$ is bijective and continuous from
$\bominft$ into $\boms \backslash \{\zero\}$ and 
$$
 \Phi^{-1}(\xs) = \frac{\xs}{r(\xs)^2},  \mbox{ for  } \xs \in  \oms \backslash \{\zero\}
$$
(although $\Phi$ can be extended to $\R^2 \backslash \{0\}$ and
can be seen as an involution,  we prefer to distinguish between $\Phi^{-1}$ and $\Phi$).  We also have
$$
\; r(\Phi(\x)) = \frac{1}{r(\x)},  \mbox{ for } \x \in \bominft, 
$$
and
\begin{itemize}
\item $\Phi(\x) = \x$ iff $\x \in {\bominft} \cap  {\boms}={\bominft} \cap  {\bomz}$.
\item $\Phi(\Tinft_i ) = S_i \backslash \{\zero\} $, $i = 1,\cdots,M$.
\end{itemize}
\begin{remark}
{\it 
Note that $M$, the number of infinite triangles,  is not a discretization parameter and is not intended to become large 
or to tend to infinity. In general, the large domains $ \Tinft_1, \cdots,  \Tinft_M$
remain fixed during mesh refinement (\eqref{decompo_inft} is indeed a domain decomposition). For example, if $\bomega \subset  [-R, R]^2$ for some $R > 0$,  one can choose the following decomposition (which will be adopted in the numerical tests):  
\begin{equation}
\om = \omz \cup \Tinft_1 \cup \Tinft_2 \cup  \Tinft_3 \cup  \Tinft_4,
\end{equation}
with
$$
\begin{array}{rcl}
\Tinft_1 &=& \{ (x, y) \in \R^2 \bve  x \geq \max(R,  |y|) \}, \\
\Tinft_2 &=& \{ (x, y) \in \R^2 \bve  y\geq  \max(R,  |x|)   \}, \\
\Tinft_3 &=& \{ (x, y) \in \R^2 \bve x \leq - \max(R,  |y|)  \}, \\
\Tinft_4 &=& \{ (x, y) \in \R^2 \bve  y\leq - \max(R,  |x|)   \},
\end{array}
$$
and $\omz = \R^2 \backslash (\overline{\omega} \cup \Tinft_1 \cup \Tinft_2 \cup  \Tinft_3 \cup  \Tinft_4)$. With such a decomposition, we have 
$$
r(x, y) = \max (\frac{|x|}{R}, \frac{|y|}{R}).
$$
}
\end{remark}
\begin{figure}
\begin{center}
\includegraphics[width=0.40\textwidth,height=5.5cm]{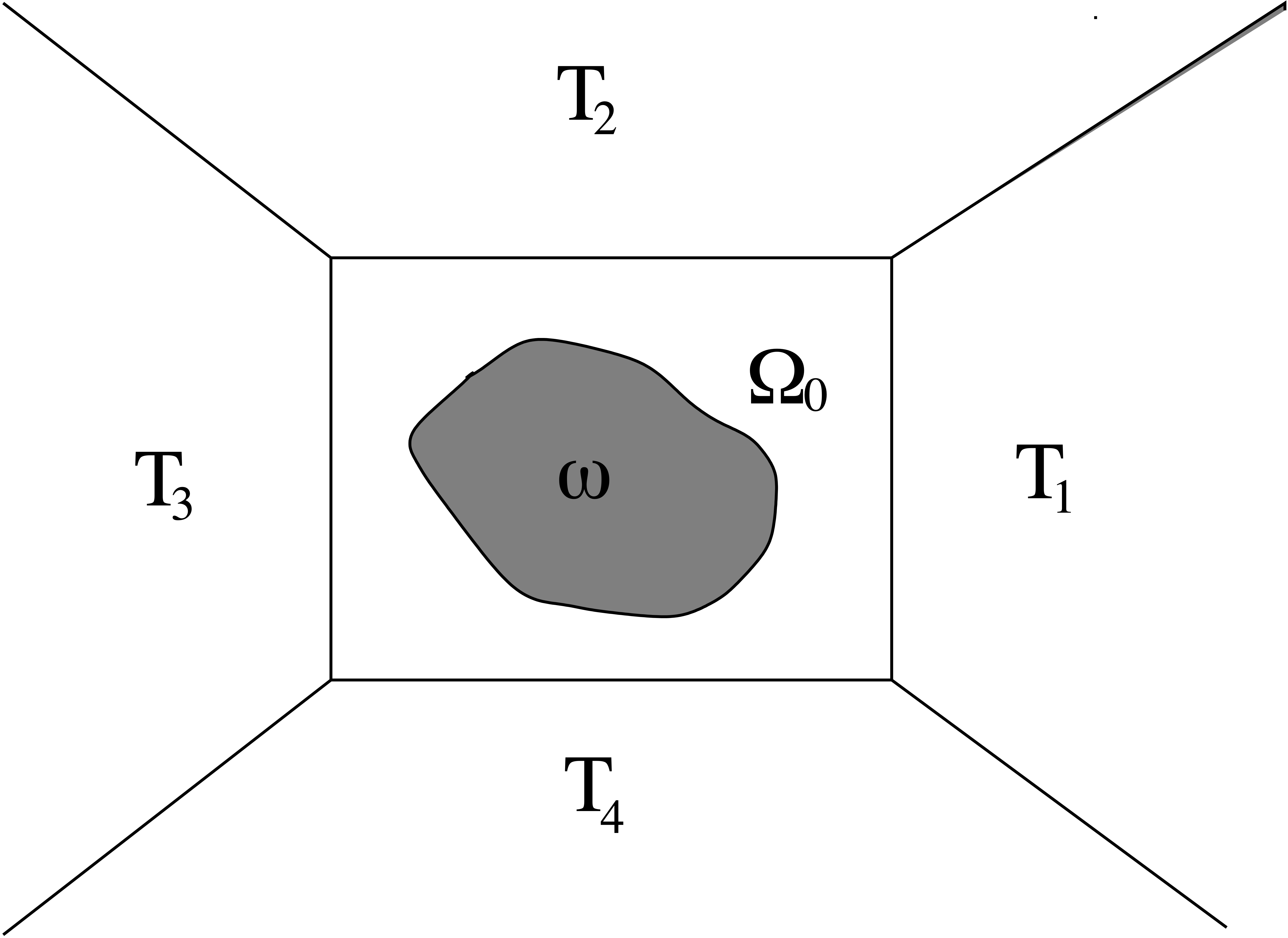}
\;\;\;
\includegraphics[width=0.40\textwidth,height=5.5cm]{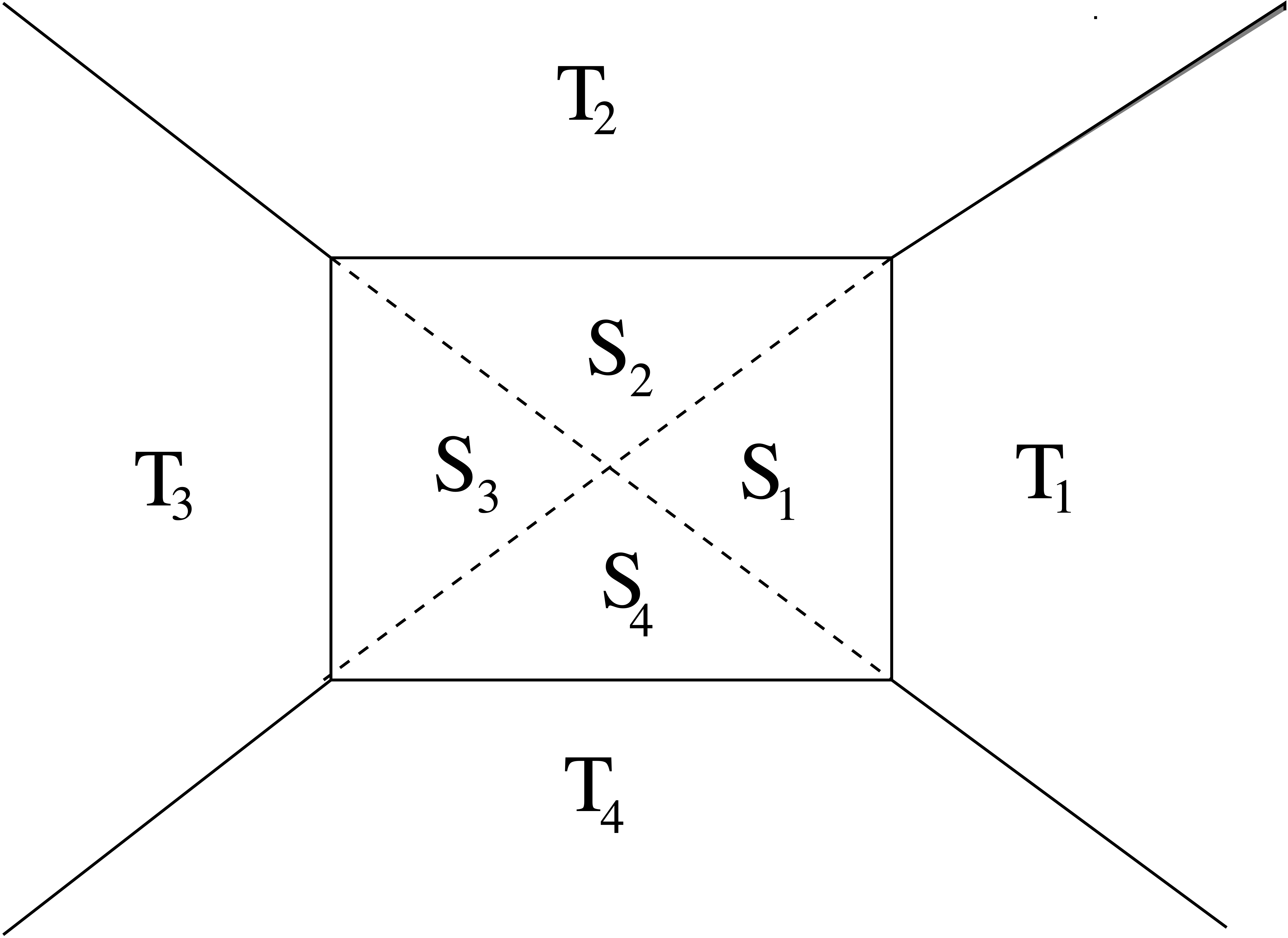}
\end{center}
\caption{(left) decomposition of the exterior domain $\om$ into the union of four infinite triangles and a bounded domain $\omz$.  (right) The four triangles $S_1, \cdots, S_4$  make up the fictitious  bounded domain $\oms$.} 
\label{decomp_illust}
 \end{figure}

Now, we introduce  a pair of triangulations $(\Tauh, \Taush)$  satisfying the following requirements:
\begin{itemize}
\item $\Tauh$ (respectively $\Taush$) is a conformal triangulation of  $\omz$ (respectively, of $\oms$)  which is shape regular in the classical sense,
\item for all $i \leq M$, the restriction of   $\Taush$ to $S_i$ is a triangulation of $S_i$,
\item the triangulations $\Tauh$ and $\Taush$ have
the same trace on $\pt \ominft \cap \pt \omz = \pt \oms \cap \pt \omz$.
\item $\Taush$ is $\mu$-graded, where $\mu \in (0, 1]$ is a fixed gradation parameter (see \cite{boulmezaoudm2an}). In other words,
$\Taush$  satisfies
\begin{equation}\label{shapereg_tauh1}
\max_{K \in \Taussh} \frac{h_K}{d^{1-\mu}_K} \leq \cst_1 h, \;
\end{equation}
\begin{equation}\label{shapereg_tauh2}
\max_{K \in \Taush \backslash \Taussh} h_K \leq \cst_2 h^{1/\mu}, \;
\end{equation}
\begin{equation}\label{shapereg_tauh3}
\min_{\Ks \in \Taussh} d_{\Ks} \geq \cst_3  h^{1/\mu},
\end{equation}
where $ \Taussh = \{ \Ks \in  \Taush \bve \zero \not \in  \Ks\}$ and 
$$
 h = \max_{K \in \Tauh \cup  \Taush }  h_K \mbox{ and } d_K =  {\rm dist}(0, K) :=\inf_{\x \in K} |\x|.
$$
Here $ \cst_1 > 0$ $ \cst_2 > 0$ and $ \cst_3$  are three constants which are independent of
the triangulation. 
\end{itemize}

The construction of triangulations satisfying conditions \eqref{shapereg_tauh1},
 \eqref{shapereg_tauh2} and  \eqref{shapereg_tauh3} is detailed in \cite{boulmezaoudm2an}, \cite{boulmezaoud_mziou1}, \cite{boulmezaoud_mziou2} and  \cite{bhowmikmziou}. We may observe that triangles of $\Taush$ are of size $h$ when they are
 adjacent to the interface $\boms\cap \bominft$, and  of  size $h^{1/\mu}$ when they touch the origine. \\
 The last item in unrolling the inverted finite element method is the functional transform
 $v \fleche \hv$ which associates to each function $v$ defined over $\bominft$, the function
 $\hv$ defined over $\boms$ as follows:
 \begin{equation}\label{definition_kelv}
 \hv({\xs})\ = r(\xs)^{-\theta+1} v(\Phi^{-1}(\xs)) \mbox{ for }  \xs  \in \boms.
\end{equation}
Here $\theta$ is a fixed real parameter.  For reasons that will become clear later, we assume that
\begin{equation}\label{inclusion_cond}
\theta > 0.
\end{equation} 
Conversely, $v$ can also be expressed
in terms of $\hv$ as follows
 \begin{equation}\label{inv_kelv}
 v(\x) = r(\x)^{-\theta+1}  \hv(\Phi({\x})),  \mbox{ for }  \x  \in \bominft.
\end{equation}
We fix now a natural number $k \geq 1$ (which will be chosen equal in the numerical tests)
and define the discrete space
$$
\begin{array}{rcl}
W_h&=&\left\{v_h\in \CC^0(\R^2 \backslash \omega)  \bve  \; \forall K\in \Tauh, \;  v_{h|K}\in \mathbb{P}_k (K),  \right. \\
&& \; \left. \forall \Ks\in \Taush, \;  \hv_{h|\Ks}\in \mathbb{P}_k(\Ks), \;  \hv_{h}(\zero) = 0 \right\}.
\end{array}
$$
Function of $W_h$ are piecewise polynomial in the FEM region $\omz$, but not in the IFEM region
$\ominft$. The latter fact is due to the distorsion caused by the inversion.  \\
In addition, observe that for $v \in W_h$,
$$
|\hv(\xs)| \leq C |\xs| \mbox{ for } \xs \in  \oms,
$$
and in view of \eqref{inv_kelv} and \eqref{estima_ray}
$$
|v(\x)| \leq C  r(\x)^{-\theta+1} r( \Phi({\x})) \leq  C |\x|^{-\theta}, \; \mbox{ for } \x \in \ominft,
$$
where $C$ is a generic constant. Similarly, we also have
$$
|\grad v(\x)| \leq C   |\x|^{-\theta-1}, \; \mbox{ for } \x \in \ominft.
$$
Since $\theta > 0$, we deduce that 
\begin{equation}\label{inclusion_wh}
W_h \subset \Wlog{1}{0}{\om}.
\end{equation}
Furthermore, $W_h$ is a finite dimensional space. When $k=1$, its dimension is given by (see Section \ref{impl_sec}): 
\begin{equation}\label{dimWh}
\dim W_h = N_i + N_b + N_i^\star + N^\star_b.
\end{equation}
where
\begin{itemize}
\item $N_i$ is the number of internal nodes of the triangulation $\Tauh$ (that is to say the nodes which are strictly inside $\omz$), 
\item $N_b$ is the number of nodes of the triangulation $\Tauh$ belonging to the boundary  $\pt \omega$,
\item $N_i^\star$ is the number of internal nodes in the triangulation $\Taush$ of $\oms$ (the node
$\zero$ is not counted),
\item $N_b^\star$  is the number  nodes which are common to the triangulations $\Tauh$ and $\Taush$ (that is, nodes belonging to $\boms \cap \bominft = \bomz \cap \bominft$).
\end{itemize}
$\;$\\
The discrete problem can be written as: find $u_h \in W_h$ such that: for all $w_h \in W_h$
\begin{equation}\label{weak_form_disc}
\int_{\om} \sigma \grad u_h. \grad w_h dx +  \left(\int_{\om} \wei(\x) u_h  dx\right)\left(\int_{\om} \wei(\x) w_h  dx\right) =  \int_{\om} f w_h dx +   \<g, w_h\>_{\pt \om}.
\end{equation}
We state this
\begin{proposition}\label{error_estima_under_asspt}
Under assumptions  \hyp{1}-\hyp{2}, the discrete problem
has one and only one solution $u_h \in W_h$. Furthermore, if  $u \in H^{k+1}_{loc}(\om)$ is such that 
\begin{equation}\label{decreasing_u}
\forall |\lambda| \leq k+1, \; \|(\pt^\lambda u)(|\x|,  \sigma)\|_{L^2(\S^1)} \leq  \frac{C}{|\x|^{\theta+|\lambda|}} \mbox{ for } |\x| \geq R,
\end{equation}
for some constants $C > 0$ and $R > 0$ are two constants, then for each $\eps \in ]0,  \min(1, \theta)[$, there exists a constant $C_\eps  > 0$, not depending on $u$ and $h$ such that
\begin{equation}\label{err_estim_theta}
\|u - u_h\|_{\Wlog{1}{0}{\om}} \leq C_\eps (h^{k} \|u\|_{H^{k+1}(\omz)} +  h^{k\min(\frac{\mu_0}{\mu}, 1)} \|u\|_{ \Wei{k+1}{k+\theta-\eps}{\ominft}}),
\end{equation}
where $\mu_0= \frac{\theta-\eps}{k}.$ If in addition $\mu < \frac{\theta}{k}$, then
\begin{equation}\label{err_estim_theta_bis}
\|u - u_h\|_{\Wlog{1}{0}{\om}} \leq C^\star  h^{k},
\end{equation}
for some constant $ C^\star$ not depending on $h$.
\end{proposition}
\begin{proof}
Well posedness of the finite dimensional problem \eqref{weak_form_disc} is straighforward
(this is a Galerkin approximation of an elliptic problem). Error estimate \eqref{err_estim_theta_bis}
follows from C\'ea's Lemma
\begin{equation}
\|u-u_h\|_{\Wlog{1}{0}{\om}} \leq C \inf_{v_h \in W_h}\|u-v_h\|_{\Wlog{1}{0}{\om}}.
\end{equation}
and estimate of the best approximation error given in \cite{boulmezaoudm2an}. The reader can refer
to \cite{bhowmikmziou} and  \cite{boulmezaoudm2an} for other details.
\end{proof}

\section{Implementation and computational results}\label{impl_sec}
Here focus is to briefly describe the implementation of the method and
to show some numerical results when $k=1$.  Let us first give a basis of $W_h$.  Let 
 $(\x_i)_{1 \leq i \in J}$ be the nodes
of  $\Tauh$ and $(\x_i)_{1 \leq i \in \pt J}$, with $\boms \cap \ominft$,  the  nodes
of  $\Tauh$ belonging to the common common boundary
 (thus, $(\x_i)_{1 \leq i \in J  \backslash \pt J}$ are internal
nodes of $\Tauh$). Similarly, let $(\xst{i})_{i \in \Js }$ be  the nodes of $\Taush$  from which
we exclude the origine and $(\xst{i})_{i \in \tJs}$  nodes of   $\Taush$ belonging to the common boundary
$\boms \cap \bominft$ (thus, $(\xst{i})_{i \in \Js  \backslash \tJs}$ are internal nodes of $\Taush$). 
 Since the triangulations $\Tauh$ and 
$\Taush$ have the same trace on $\boms \cap \ominft$, the sets 
$\{ \x_{i} \bve i \in  \pt J\}$ and 
$\{ \xst{i} \bve i \in \tJs\}$ are equal. \\
Define $\bw_i$, $i \in J$, as the unique function
of $W_h$ satisfying
$$
\bw_i(\x_j) = \delta_{i, j} \mbox{ for all } j \in J, \; \hbw_i(\xst{j}) = 0 \mbox{ for all } j \in \Js  \backslash  \tJs.
$$
Define also $\bws_i$, $i \in \Js$, as the unique function of $W_h$ satisfying
$$
\bws_i(\x_j) = 0 \mbox{ for all } j \in  J  \backslash \pt J, \; \hbws_i(\xst{j}) = \delta_{i, j} \mbox{ for all } j \in \Js.
$$
We may notice that
\begin{itemize}
\item  $\bw_i =0$ in $\bominft$ for all $i \in  J   \backslash  \pt J$. 
\item $\bws_i = 0$  in $\bomz$ for $i \in  \Js  \backslash  \tJs$.
\item If  $\x_i \in \boms \cup \bominft$, $i \in J$, then the support of $\bw_i$ lies straddles the two domains 
$\bomz$ and $\bominft$.
\end{itemize}
The functions $(\bw_i)_{i \in J}$ and $(\bws_i)_{i \in  \Js  \backslash  \tJs}$ form a basis of $W_h$ and
for all $v_h \in V_h$  we can write
\begin{align}\label{f:linear_model001a}
v_h(\x) &=\sum_{i \in J} v_h (\x_i)\ \bw_i(\x) +  \sum_{i \in \Js \backslash  \tJs}  \widehat{v}_h (\xst{i}) \bws_i(\Phi(\x)), \; \mbox{ for } \x \in \om.
\end{align}
Using this basis, one can write the discrete problem \eqref{weak_form_disc} into the form $A U = B$  with
$A$ the stifness matrix whose entries are integrals of the form
\begin{equation}
(A)_{i, j} := \int_{\om}   \sigma (\x) \grad \psi_i . \grad\psi_j dx  = \int_{\omz}   \sigma (\x) \grad \psi_i . \grad\psi_j dx + \int_{\ominft}   \sigma (\x) \grad \psi_i . \grad\psi_j dx,
\end{equation}
where $\psi_i$ and $\psi_j$ are two basis functions.  The integral on $\omz$ on the right hand side
can be computed In the same way as in the finite element method. The second integral is over 
the unbounded subdomain $\ominft$ and can be transformed to a (singular) integral on the
bounded domain $\oms$ by means of a change of variables. In \cite{bhowmikmziou} it is proven that
$$
\begin{array}{l}
  \dps{ \int_{\ominft} \sigma(\x)(\nabla \psi_i)^T \nabla \psi_j\ dx }  =  \dps{ \int_{\oms}     (\nabla \psihat_i)^T \Matstar(\xs)      (\nabla \psihat_j)  d{\xstar} }\\
 \dps{  + \int_{\oms} \bfs(\xs)   \psihat_i(\xs)  \psihat_j(\xs)   d{\xstar}  }\\
 \dps{ +   \int_{\oms}   \cfs(\xs)      [\psihat_i(\xs) \nabla \psihat_j(\xs) + \psihat_j(\xs) \nabla \psihat_i(\xs)]^T  d{\xstar}.}
  \end{array}
$$
with
$$
\begin{array}{rcl}
\Matstar(\xs) &=&  \dps{  \sigma(\Phi(\xs)) \frac{r(\xs)^{2 \theta-4}}{\nrm{\vh}^2}   \left(\nrm{\vh}^2 r(\xs)^2I -   2 r(\xs) (\vh \xs^T + \xs \vh^T)  + 4\xs {\xs}^T\right)}, \\
\bfs(\xs)   &=& \dps{  - (\theta-1)^2 \sigma(\Phi(\xs)) \frac{ r(\xs)^{2 (\theta-2)}}{\nrm{\vh}^2}},  \\
\cfs(\xs) &=& \dps{  - (\theta-1) \sigma(\Phi(\xs)) \frac{  r(\xs)^{2 \theta-3} }{ \nrm{\vh}^2} ( \vh - \frac{2 \xs}{r(\xs)}), }
\end{array}
$$
where $\vh = \vh (\xs)$ is the altitude vector corresponding to  $\Tinft_1,\cdots,\Tinft_M$  (see  \cite{bhowmikmziou} or \cite{boulmezaoudm2an}).  In the latter formulas, the vectors $\xs$, $\nabla \psi_i$ and $\vh$ are considered as column ones.  \\
These integrals are then computed by means of quadrature rules on triangles, as explained
in \cite{bhowmikmziou} and \cite{boulmezaoudm2an}  (see also
\cite{hammer1, hammer2}, \cite{dunavant}, \cite{silvestre_q},  \cite{strangfix} or \cite{zienkiewicz_book}).   One can also use an affine mapping to transform the integrals
on a triangle $K$ to the standart triangle $ {\hat{K}} = \{(x, y) \in \R^2 \bve x \geq 0, \; y \geq 0, \; x+ y \leq 1\},$ or 
 the map $(\xi, \eta) \in [0, 1]^2 \mapsto  (\xi, \eta(1-\xi))  \in  {\hat{K}}$ which transforms the square $[0, 1]^2$
 into $ {\hat{K}}$ (allowing the use of tensorized quadrature formulas). \\
All the computational experiences which follow are carried out with Matlab software.  The purpose of these numerical experiences is to verify that the method is efficient with this Neumann problem, when the coefficients are constant or even variable to infinity.  The impact of the mesh gradation, already studied many times (see \cite{boulmezaoudm2an}, \cite{bhowmikmziou}, \cite{boulmezaoud_mziou1} and   \cite{boulmezaoud_mziou2} and  \cite{boulmezaoud_divcurl_ifem}), will not be studied again here. \\
$\;$\\
 In all the following simulations $\omega$ is the unit ball of $\R^2$   and unless otherwise indicated, we choose: 
\begin{equation}\label{value_pratique_theta}
 \theta = 1.01.
\end{equation}

$\;$\\
{\it Example 1. } 
We consider the system
$$
- \Delta u = f \mbox{ in } \R^2 \backslash \bomega, \; \frac{\pt u}{\pt n} = 0 \mbox{ on } \pt \omega,
$$
with $f$ chosen such that the exact solution is given by
\begin{equation}\label{essai_neum}
 u(x, y) = \frac{x}{\sqrt{x^2+y^2}} \sin\left(\frac{\pi}{2} \frac{1}{(x^2+y^2)^{2}}\right).
 \end{equation}
 The long distance behavior of this function is as follows:
 \begin{equation}\label{essai_neum_ex1}
 |u(x, y)|  \sim \frac{c_1}{r^4}, \;  |\grad u| \sim \frac{c_2}{r^5},  \;  |D^2 u| \sim \frac{c_3}{r^6}, 
 \end{equation}
 with $r = \sqrt{x^2+y^2}$.  It follows that $u \in H^2_{loc}(\om)$ and the assumption  \eqref{decreasing_u} is satisfied
  with $k= 1$.   According to Proposition \ref{error_estima_under_asspt},  one must have 
\begin{equation}\label{err_estim_example1}
\|u - u_h\|_{\Wlog{1}{0}{\om}} \leq C_1  h,
\end{equation}
 regardless of the value chosen by the parameter $\mu \in (0, 1]$. \\
  Table \ref{tab_ex1_Neumann}  summarizes the numerical results obtained with this example. In particular, it includes the values of the relative weighted   $L^2$-error on the solution $u$  and the $L^2$-error on its gradient and this for different values of the mesh size $h$ and of the gradation parameter $\mu$ ($\mu = 1, 0.75$ or $0.5$). In Figure \ref{exmp1_fig_errors}    the curves of these errors are displayed versus the mesh size $h$ in a logarithmic scale.  As expected, the  two errors decrease with respect to $h$ and we can observe that the value of gradation parameter $\mu$  has no significant effect on the errors (curves are quasi-superposed).   In other words, since the solution decreases sufficiently at large distances, that is when $r \to +\infty$,   it is not necessary to grade the inverted mesh in order to capture the behavior of the solution at infinity; the method is efficient 
  regardless of the chosen value of the gradation mesh $\mu$.  This in accordance with forecasts of Proposition \ref{error_estima_under_asspt}.  Furthermore, by calculating the slopes of these curves we observe that   the weighted $L^2$ error on the solution decreases as $h^2$  while the $L^2$ error on the gradient is of the order of $h$. Although this confirms estimate \eqref{err_estim_example1}, there is a superconvergence of the 0th order error in the (weighted) $L^2$-norm, as in the finite element method in bounded domains. We have no theoretical proof of this observation which is specific to this example, since, contrary to the method of finite elements (see, e. g.,  \cite{ciarlet_fem}), this superconvergence of the (weighted) $L^2$-error does not occur with other problems (see \cite{boulmezaoudm2an},   \cite{bhowmikmziou}, \cite{boulmezaoud_mziou1}, \cite{boulmezaoud_mziou2},  and \cite{boulmezaoud_divcurl_ifem}) nor in the case of the 3rd example which will follow.  \\
 Figure  \ref{example1_solutions_images}  shows the 2D view of the exact solution  and the approximate solutions for $\mu=1$. 
 We can observe that the approximate solutions is close to the exact one; the two solutions are difficult to distinguish 
 with the naked eye. 
\begin{table}[!h]
\begin{center}
\begin{tabular}{|r|c|ccccc|}
\hline
 \hline
$\mu$ &  \textbf{Mesh size   h} 	&     0.65 &       0.32  &  0.21 &   0.15 &     0.07		\\					
 \hline			
   $1$&  Wghd. mean value   &  -2.3e-5 &  4.6e-5 & 8.3e-5 & 2.4e-5 &  4.7e-6\\
&\textbf{Rel.  $L^2_\wei$-error}  	 &  {\bf   0.166 }&        {\bf  0.034 }  &    {\bf  0.015 }   &  {\bf   0.012  }&      {\bf     0.003 }\\
 &  in fem region   	         & 0.164  &   0.034  &  0.015    & 0.012 &  0.003 \\			
&  in ifem region  			& 0.233  &    0.014  &  0.004  &  0.003   &  0.0008 \\		
& \textbf{Rel. $L^2$-error on $\grad u$}   &   {\bf  0.716 }  &       {\bf  0.238 }  &   {\bf   0.163  }&  {\bf   0.147 }  &     {\bf   0.074 }\\				
 &in fem region &    0.734  &   0.243 &   0.168  &  0.152   &   0.076	\\
& in ifem region&       0.298  &   0.152  &  0.046 &   0.061 &      0.037  \\			
  \hline			 	
  $0.75$ &Wghtd. mean value   &  -4.4e-5 &  4.6e-5 & 8.15e-5 & 2.3e-5 & 4.6e-6\\
& \textbf{Rel.  $L^2_\wei$-error} 	 &      {\bf  0.167 }   &      {\bf  0.033  }  &   {\bf  0.015 }   &   {\bf  0.012 }  &   {\bf    0.003 }  \\			
& in fem region  	           &  	0.164  &   0.033  &  0.014  &  0.012 &     0.003\\			
& in ifem region  			& 0.249   & 0.015   & 0.004   & 0.002  &   0.001 \\		
& \textbf{Rel. $L^2$-error on $\grad u$}   &  {\bf    0.716  } &     {\bf   0.238 } &     {\bf  0.163  }&     {\bf  0.147  }&   {\bf     0.074 }\\					
& in fem region &  	 0.733  &   0.242   & 0.167  &  0.151   & 0.075 \\
& in ifem regin&         0.312  &    0.167 &   0.051 &   0.068 &   0.041 \\	
  \hline
   $0.5$ & Wght. Mean value  & -6.9e-5  & 4.6e-5 & 7.9e-5&  2.3e-5& 4.4e-6\\					
& \textbf{Rel.  $L^2_\wei$-error}   	 &  {\bf   0.167 }  &     {\bf   0.033  } &    {\bf  0.014  } &  {\bf    0.012 }  &      {\bf    0.003 }	 \\			
& in fem region     	         & 0.164   & 0.033  &  0.014  &  0.012  &   0.003	 \\			
& in ifem region    			&  0.275   &   0.029  &  0.005  &  0.005  &    0.002\\		
& \textbf{Rel. $L^2$-error on $\grad u$} &     {\bf   0.717 }  &    {\bf    0.240 } &  {\bf     0.163  }&   {\bf    0.148 }&    {\bf     0.074 }\\					
 &in fem region &   	0.733  &   0.242   & 0.167   & 0.151   &   0.075\\
& in ifem region &     0.345  &    0.201  &  0.064   & 0.085   & 0.052  \\
\hline
\end{tabular}
\caption{ (Example 1) Global and local relative weighted $L^2$ errors on $u$ and relative $L^2$ error on $\grad \u$ for several values of $\mu$.  Here $ \omz   = ]-1.5,1.5[^2 \backslash \bomega$ (FEM region) and $\ominft =\R^2  \backslash   [-1.5,1.5]^2$ (IFEM region). }\label{tab_ex1_Neumann}
\end{center}
\end{table}
\begin{figure}
\begin{center}
\includegraphics[width=0.49\textwidth,height=6.5cm]{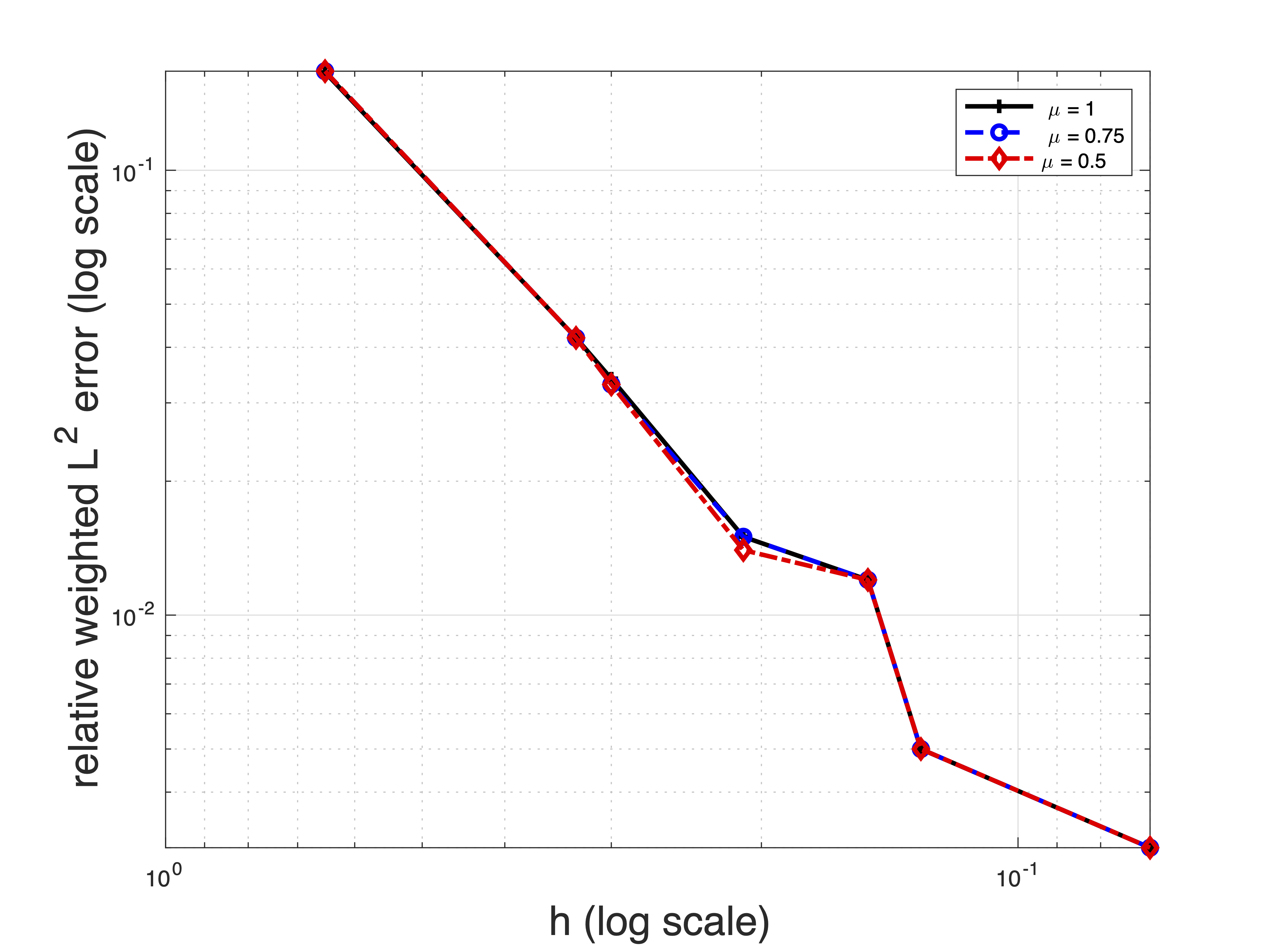}
\includegraphics[width=0.49\textwidth,height=6.5cm]{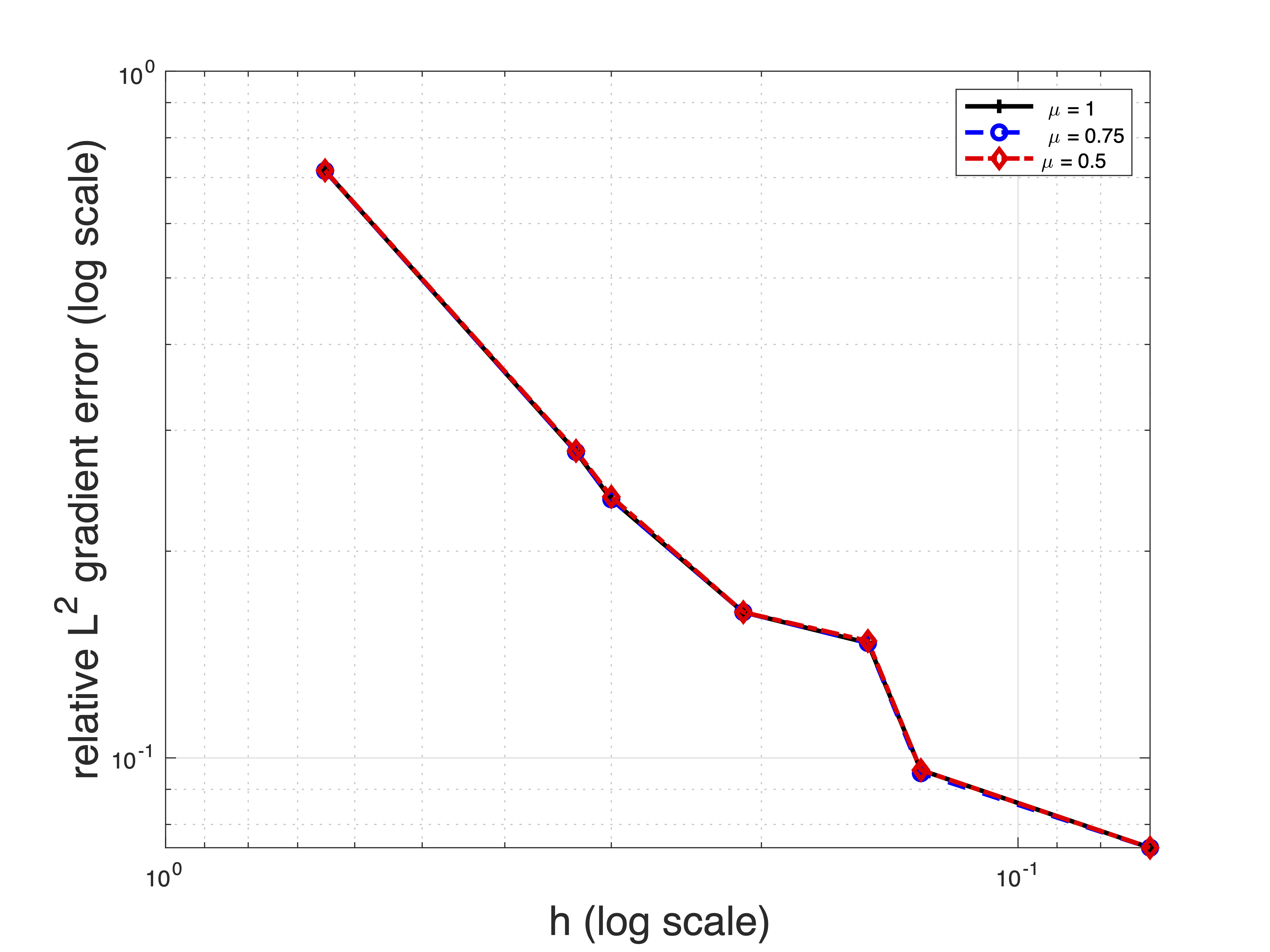}
\end{center}
\caption{ (example 1) global relative weighted $L^2$ error on $u$  (left figure) and  global relative $L^2$ error on $\grad \u$ (right figure). }
\label{exmp1_fig_errors}
 \end{figure}
 \begin{figure}
\begin{center}
\includegraphics[width=0.49\textwidth, height=5.5cm, width = 7cm]{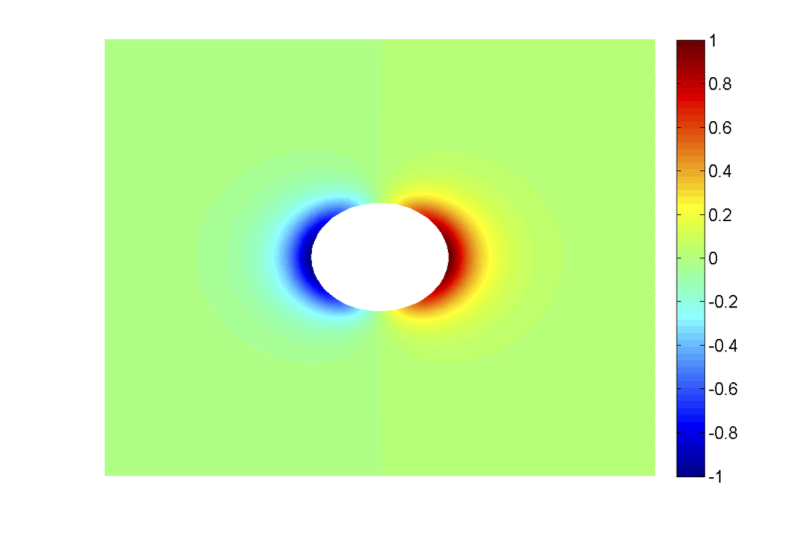}
\includegraphics[width=0.49\textwidth, height=5.5cm,  width =7cm]{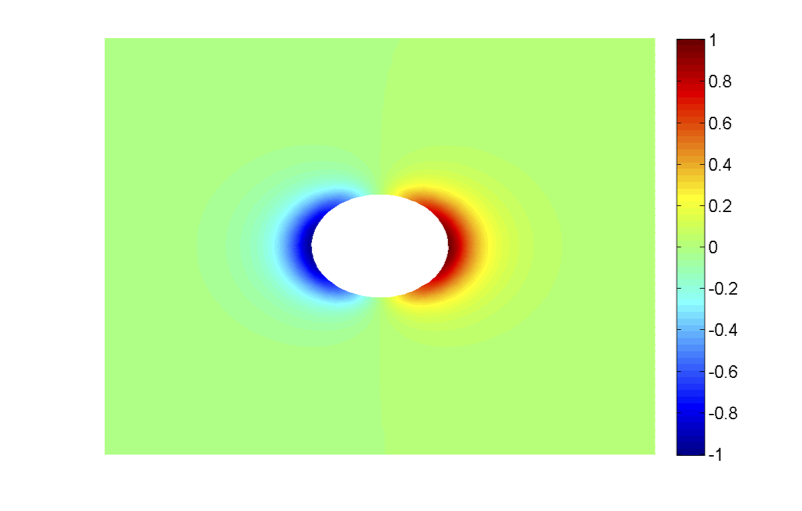}
\end{center}
\caption{(example 1) exact solution (left) and approximate solutions  (right) with $\mu = 1$.}
\label{example1_solutions_images}
 \end{figure}

$\;$\\
{\it Example 2. } 
We consider the  system \eqref{second_ord_equa}-\eqref{neumann_code} with a coefficient  $\sigma$ varying to infinity:  
$$
\sigma(x, y) =1 - \frac{x^2-y^2}{2(x^2+y^2)}, 
$$
and   $f$ chosen such that the exact solution is 
\begin{equation}\label{ex2_essai_neum}
 u(x, y) = \frac{x}{\sqrt{x^2+y^2}} \sin\left(\frac{\pi}{2} \frac{1}{(x^2+y^2)^{2}}\right).
 \end{equation}
 Here also the solution has  a behavior of type \eqref{essai_neum_ex1} at large distances, and according to \eqref{err_estim_theta_bis}, the error estimate \eqref{err_estim_example1}, remains true, independently of the value taken by $\mu$. \\
 The numerical results illustrated in Table \ref{tab_ex2_Neumann} and  Figure \ref{exmp2_fig_errors} are analogous to those obtained in the first example, although the coefficient $\sigma$ is variable here up to infinity, that is to say without being constant  at large distances.  This again confirms the usability of the method even with variable coefficients, i. e. when one cannot rely on an explicit knowledge of the fundamental solution to make an integral formulation or a series expansion in the farthest region.
 
\begin{table} \label{tableExample2}
\begin{center}
\caption{Global relative weighted $L^2$ error on $u$  (left figure) and  global relative $L^2$ error on $\grad \u$ (right figure).}

\begin{tabular}{|r|c|cccccc|}
\hline
 \hline 
$\mu$ & \textbf{Mesh size   h}             &        0.65  &   0.30  & 0.21   & 0.15 & 0.13& 0.07\\
  \hline
$1$   &  Wght. mean value   & 1.7e-2&  1.1e-4 & 3.7e-4 & 7.8e-5 & 2.8e-5 &  1.01e-5 \\
   &\textbf{Rel.  $L^2_\wei$-error}        &{\bf 0.942  }      &{\bf    0.106  }   &{\bf  0.037 }   &{\bf   0.022 }   &{\bf 0.011  }&{\bf  0.007 }  \\
   &                in fem region       &0.890    &   0.101   & 0.037  &  0.021 & 0.011   & 0.007\\
    &               in ifem region      &2.465   &   0.230   & 0.082   & 0.037 &0.022    & 0.013\\
   & \textbf{Rel. $L^2$-error on $\grad u$}       &{\bf 0.984  }     & {\bf  0.247  }     & {\bf 0.165 }   &{\bf 0.148 }    &{\bf 0.096 }      & {\bf0.074 }   \\
    &               in fem region       &0.967    &   0.250   & 0.170   & 0.152 & 0.099   & 0.076\\
      &             in ifem region      &1.236   &   0.185   & 0.060    & 0.064 & 0.025  &  0.037\\
    \hline  
 \hline
$0.75$   & Wght. mean value  &1.7e-2  & 1.2e-4  & 3.7e-4  & 7.7e-5   & 2.4e-5 & 9.9e-6\\
   & \textbf{Rel.  $L^2_\wei$-error}        &{ \bf 0.946  }    &{ \bf 0.105    }   &{ \bf  0.039  }      &{ \bf 0.022 }      &{ \bf 0.011 } &{ \bf  0.007 }  \\
   &                in fem region       &0.893   & 0.101     & 0.037      & 0.021     & 0.010 & 0.006\\
   &                in ifem region      & 2.486   & 0.224     & 0.081    &  0.035    & 0.022 & 0.013\\
   & \textbf{Rel. $L^2$-error on $\grad u$}        &{ \bf 0.984   }   &{ \bf  0.247  }      &{  \bf 0.165     }     &{ \bf 0.148  }       &{  \bf 0.096  }  &{  \bf 0.074 }  \\
   &                in fem region       &0.967  & 0.250     & 0.169      &  0.152      & 0.099 & 0.076\\
    &               in ifem region      & 1.240    & 0.198    & 0.064     & 0.070      & 0.028& 0.041\\
    \hline
    \hline
$0.5$ & Wght. mean value &  1.7e-2   & 1.2e-4 & 3.6e-4 & 3.6e-4 & 2.4e-5 & 9.8e-6\\
   & \textbf{Rel.  $L^2_\wei$-error}     & {\bf 0.948 }  &{ \bf  0.105 }     &{  \bf 0.038 }     &{ \bf 0.022   }  & { \bf 0.011  }   &{  \bf 0.007  }  \\
   &                in fem region     & 0.895    & 0.101   & 0.037  & 0.021 & 0.0105 &  0.006\\
   &                in ifem region     & 2.493     &  0.212   & 0.080   &  0.033 & 0.022 & 0.012\\
   &  \textbf{Rel. $L^2$-error on $\grad u$}     &{  \bf 0.985    }  &{  \bf 0.249  }     &{  \bf 0.165  }    &{  \bf 00.149 }   &{  \bf 0.096 }   &{ \bf  0.075 }  \\
     &              in fem region      & 0.968     & 0.250    & 0.169   &  0.152  & 0.099 &  0.076\\
   &                in ifem region     & 1.248    & 0.228    & 0.079  & 0.087  & 0.034 & 0.052\\
\hline
\end{tabular}
\caption{(example 2) global and local relative weighted $L^2$ errors on $u$ and relative $L^2$ error on $\grad \u$ for several values of $\mu$.  Here $\omz = ]-1.5,1.5[^2 \backslash \bomega$ (FEM region) and $\ominft =\R^2  \backslash   [-1.5,1.5]^2$ (IFEM region).  }\label{tab_ex2_Neumann}

\end{center}
\end{table}
\begin{figure}
\begin{center}
\includegraphics[width=0.49\textwidth,height=5.5cm]{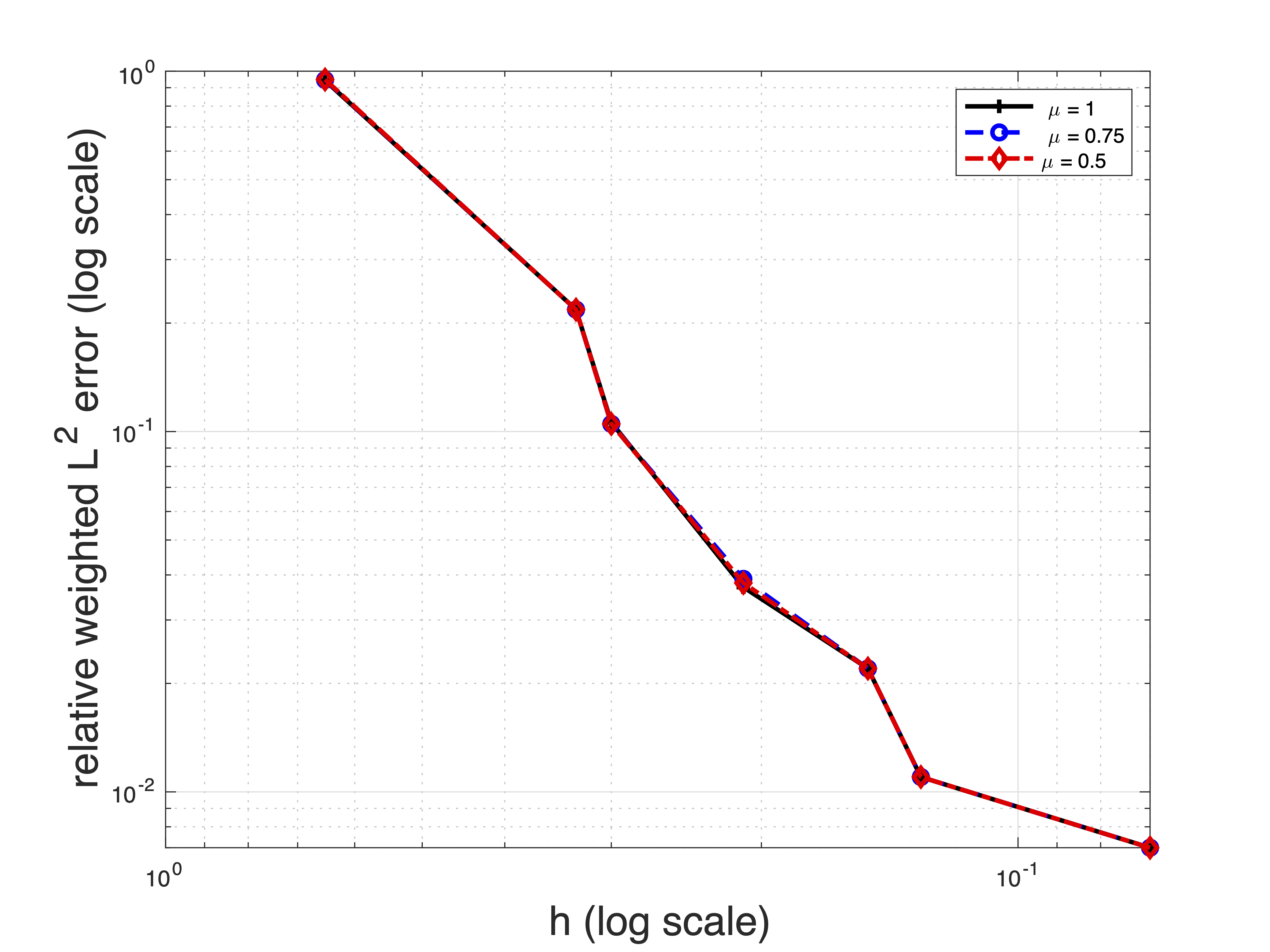}
\includegraphics[width=0.49\textwidth,height=5.5cm]{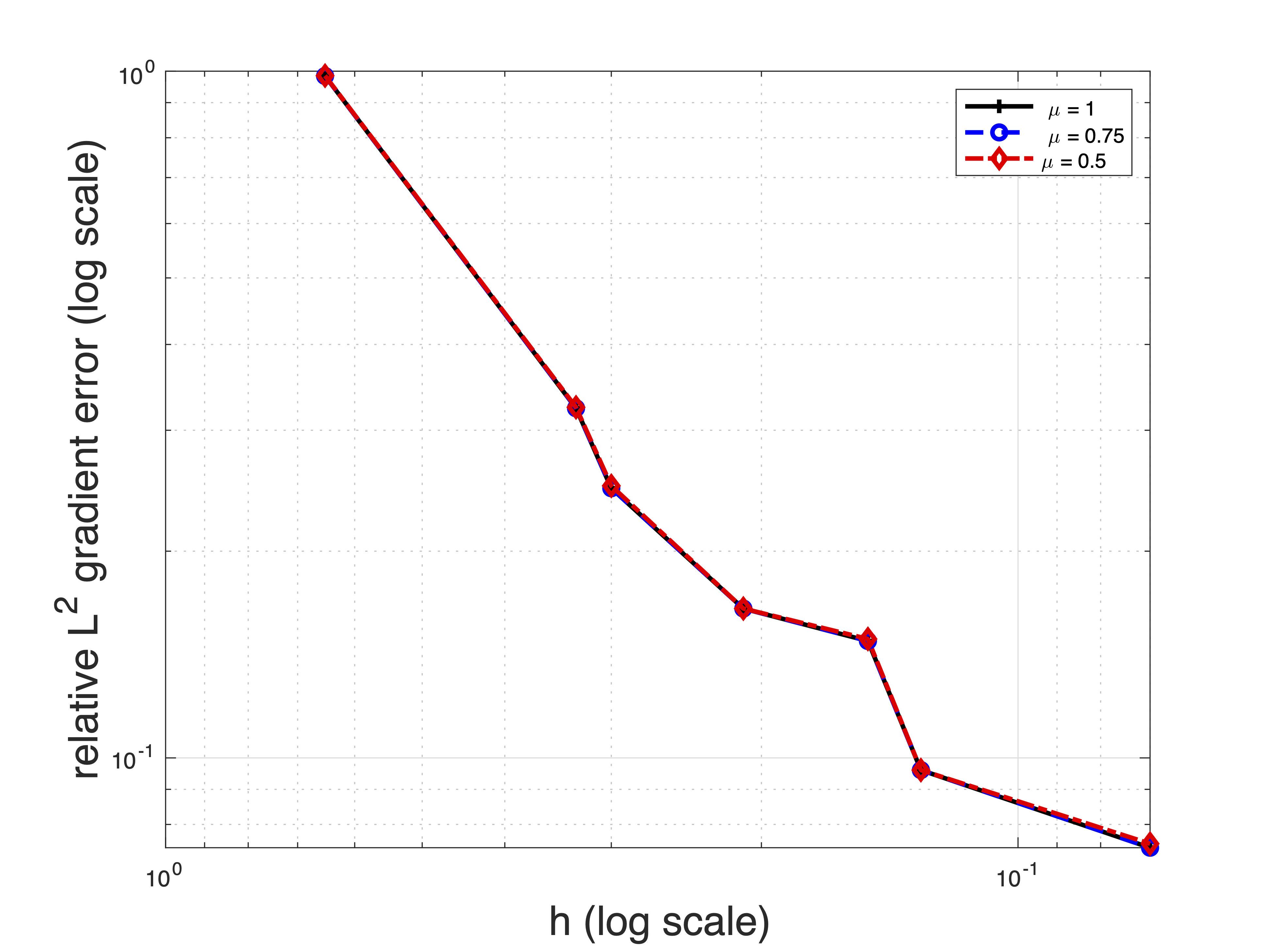}
\end{center}
\caption{(exemple 2) global relative weighted $L^2$ error on $u$  (left figure) and  global relative $L^2$ error on $\grad \u$ (right figure) versus $h$. }
\label{exmp2_fig_errors}
 \end{figure}

\section{Conclusion}
As a conclusion,  it has been demonstrated that the inverted finite element method easily adapts to second-order problems with Neumann boundary conditions in two-dimensional exterior domains. The numerical results it gives confirm its effectiveness and  the behavior of the error is in agreement with the theoretical predictions, even when the coefficient of the equations is varying over large distances and up to infinity. Furthermore, the method provides very good results  regardless of the selected value  of  the parameter $\mu$. Indeed, in the examples we have considered here the gradation does not affect the rapidity of the algorithm whose both relative weighted $L^2$ errors on $u$ and  relative $L^2$ error on $\grad \u$  are decreasing with respect to the mesh size $h$. However, from a qualitative point, this behavior turns out to be mainly the same as when we impose Dirichlet type conditions. Only the theoretical formulation of the problem is different. \\
$\;$\\
{\bf Acknowledgment.} \\
N. Kerdid, S.K. Bhowmik and S. Mziou gratefully acknowledge the support of the National Plan for Science, Technology and Information (MAARIFAH), King Abdulaziz City for Science and Technology, KSA, award number 12-MAT2996-08. \\

{\bf Declarations.}\\
$\;$\\
{\it Conflicts of interests}: The authors declare no competing interests.

\bibliographystyle{plain}
\bibliography{biblio_ifem_2022}

\end{document}